\documentclass[12pt,A4]{amsart}
\usepackage{amsmath}
\usepackage{amssymb}
\usepackage{bbm}
\usepackage{euscript}

\theoremstyle{definition}
\newtheorem{df}{Definition}[section]
\newtheorem{nota}[df]{Notation}
\theoremstyle{theorem}
\newtheorem{thm}[df]{Theorem}

\newtheorem{lem}[df]{Lemma}
\newtheorem{prop}[df]{Proposition}
\newtheorem{cor}[df]{Corollary}
\newtheorem{rem}[df]{Remark}

\newcommand{\la}{\langle}
\newcommand{\ra}{\rangle}
\newcommand{\tensor}{\otimes}
\newcommand{\Z}{\mathbb{Z}}
\newcommand{\R}{\mathbb{R}}
\newcommand{\C}{\mathbb{C}}
\newcommand{\Q}{\mathbb{Q}}
\newcommand{\W}{\mathcal{W}}
\newcommand{\al}{\alpha}
\newcommand{\be}{\beta}
\newcommand{\Com}{\mathrm{Com}}
\newcommand{\aut}{\mathrm{Aut}\,}

\begin{document}

\title[Weyl groups and VOA]{Weyl groups and vertex operator algebras generated by Ising vectors satisfying $(2B,3C)$ condition}

 \author{Hsian-Yang Chen}
  \address[H.Y.  Chen]{ Institute of Mathematics, Academia Sinica, Taipei  10617, Taiwan}
\email{hychen@math.sinica.edu.tw}
 \author{Ching Hung Lam}
 \address[C.H. Lam]{ Institute of Mathematics, Academia Sinica, Taipei  10617, Taiwan
 and National Center for Theoretical Sciences, Taiwan}
\email{chlam@math.sinica.edu.tw}
\thanks{Partially supported by NSC grant
  100-2628-M-001005-MY4 }
\date{\today}
\maketitle

\begin{abstract}
    In this article, we construct explicitly certain moonshine type vertex operator algebras  generated by a set of Ising vectors $I$ such that
    \begin{itemize}
        \item[(1)] for any $e\neq f\in I$, the subVOA $\mathrm{VOA}(e,f)$ generated by $e$ and $f$ is isomorphic to either $U_{2B}$ or $U_{3C}$; and
        \item[(2)] the subgroup generated by the corresponding Miyamoto involutions $\{\tau_e|\,e\in I\}$ is isomorphic to the Weyl group of a root system of type $A_n$, $D_n$, $E_6$, $E_7$ or $E_8$.
    \end{itemize}
The structures of the corresponding  vertex operator algebras and their Griess algebras are also studied.  In particular, the central charge of these vertex operator algebras are determined.
\end{abstract}

\section{Introduction}

The study of Ising vectors or simple conformal vectors of central charge $1/2$ was initiated by Miyamoto \cite{miy}. He noticed
that one can define certain involutive automorphisms of a VOA by using Ising vectors and the corresponding fusion rules. These
automorphisms are often called Miyamoto automorphisms or Miyamoto involutions. When $V$ is the famous Moonshine VOA
$V^{\natural}$, Miyamoto \cite{M} also showed that there is a $1$-$1$ correspondence between the $2A$ involutions of the
Monster group and Ising vectors in $V^{\natural}$ (see also \cite{ho}). This correspondence turns out to be very useful for
studying some mysterious phenomena of the Monster group and many problems about $2A$-involutions of  the Monster group
can be translated into questions about Ising vectors. In \cite{LYY,LYY2}, McKay's observation on the  affine $E_{8}$-diagram has
been studied using Miyamoto involutions. In particular, certain VOAs generated by $2$ Ising vectors were constructed explicitly.
There are nine  cases  and they are denoted by $U_{1A},U_{2A},U_{2B},U_{3A},U_{3C},U_{4A},U_{4B},U_{5A}$, and $U_{6A}$
because of their connection to the 6-transposition property of the Monster group (cf. \cite[Introduction]{LYY}), where
$1A,2A,\dots,6A$ denote certain conjugacy classes of the Monster (see~\cite{ATLAS}).

In \cite{Sa}, Griess algebras generated by 2 Ising vectors contained in a moonshine type VOA $V$ over $\R$ with a positive definite
invariant form are classified. There are also nine possible cases, and they correspond exactly to the Griess algebra
$\mathcal{G}U_{nX}$ of the nine VOAs $U_{nX}$, for $nX\in\{1A,2A,2B,3A,3C,4A,4B,5A,6A\}$. In addition, he proved that
Miyamoto involutions associated to Ising vectors satisfy a $6$-transposition property, namely,
\[
|\tau_e\tau_f|\leq 6 \qquad \text{ for any Ising vectors } e,f \in V.
\]

Motivated by the work of Sakuma \cite{Sa}, Ivanov
axiomatized the properties of Ising vectors and introduced the notion of Majorana representations for finite groups in his book~\cite{Iva}. In particular,  the famous $196884$-dimensional Monster Griess algebra constructed by Griess~\cite{G} is a Majorana representation of the Monster simple group. Ivanov and his research group also initiated a program on classifying the Majorana representations for various finite groups \cite{Iva2,Iva3,IS,IPSS}.  It turns out that most known examples of Majorana representations are constructed as certain subalgebras of the Monster Griess algebra.

One fundamental question of the theory is to determine what kinds of (6-transposition) groups can be generated by Miyamoto involutions and what are the structures of the corresponding Griess algebras and vertex operator algebras. Up to now, almost all known examples are related to the Monster simple group $\mathbb{M}$ and the Moonshine VOA $V^\natural$.  Therefore, it is very interesting to see if there are non-trivial examples which are not contained in the Moonshine VOA. In this article, we construct a series of examples associated to the Weyl groups. The main result is an explicit construction of certain moonshine type vertex operator algebras  generated by a set of Ising vectors $I$ such that
    \begin{itemize}
        \item[(1)] for any $e\neq f\in I$, the subVOA $\mathrm{VOA}(e,f)$ generated by $e$ and $f$ is isomorphic to either $U_{2B}$ or $U_{3C}$; and
        \item[(2)] the subgroup generated by the corresponding Miyamoto involutions $\{\tau_e|\,e\in I\}$ is isomorphic to the Weyl group of a root system of type $A_n$, $D_n$, $E_6$, $E_7$ or $E_8$.
    \end{itemize}
The structures of the corresponding  vertex operator algebras and their Griess algebras are also studied.   In particular, the central
charge of these vertex operator algebras are determined. Our construction  in some sense  is an $E_8$-analogue of \cite{LSa} and
\cite{LSY}, in which certain VOAs generated  by Ising vectors associated to norm $4$ vectors of the lattice $A_1\otimes R\cong
\sqrt{2}R$ for some root lattice $R$ are considered.

The organization of this article is as follows. In Section 2, we recall the basic notion of lattice VOA and the definitions of Ising vectors and Miyamoto involutions.  In particular, we recall some basic properties of the so-called $2B$ and $3C$-algebras from \cite{LYY}. We also review the construction of certain Ising vectors in lattice VOA from \cite{dlmn}. In Section 3, we give an explicit construction of  a certain VOA generated by Ising vectors such that the subgroup generated by the
corresponding Miyamoto involutions is isomorphic to the Weyl group of a root system of type $A$, $D$ or $E$. The main idea is to consider a certain integral lattice which has many sublattices isometric to $\sqrt{2}E_8$. In Section 4, we study the Griess algebra generated by the Ising vectors and determine the central charge of the corresponding VOA. In Section 5, we give a more detailed study for the type A case. In addition,  we also give a construction of a VOA such that the subgroup generated by the corresponding  Miyamoto involutions has the shape $3^{n-2}:S_n$ if $n=0\mod 3$ or $3^{n-1}:S_n$ if $n\neq 0 \mod 3$.

\section{Preliminary}
    Our notation for the lattice vertex operator algebra
    \begin{equation}
        V_L=M(1)\otimes \mathbb C\{L\}
    \end{equation}
    associated with a positive definite even lattice $L$ is standard \cite{FLM}.
    In particular, $\mathfrak h=\mathbb C\otimes_{\mathbb Z} L$ is an abelian Lie algebra and we extend the bilinear form to $\mathfrak h$ by $\mathbb C$-linearity.
    The set  $\hat{\mathfrak h}=\mathfrak h\otimes\mathbb C[t,t^{-1}]\oplus\mathbb Ck$ is the corresponding affine algebra and $\mathbb Ck$ is the $1$-dimensional center of $\hat{\mathfrak h}$.
    The subspace $M(1)=\mathbb C[\alpha_i(n)|\ 1\leq i\leq d, n<0]$ for a basis $\{\alpha_1,\cdots,\alpha_d\}$ of $\mathfrak h$, where $\alpha(n)=\alpha\otimes t^n$, is the unique irreducible $\hat{\mathfrak h}$-module such that $\alpha(n)\cdot 1=0$ for all $\alpha\in\mathfrak h$, $n$ nonnegative, and $k=1$.
    Also, $\mathbb C\{L\}=\mathrm{Span}\{e^\beta|\ \beta\in L\}$ is the twisted group algebra of the additive group $L$ such that $e^\beta e^\alpha=(-1)^{\langle \alpha,\beta\rangle}e^\alpha e^\beta$ for any $\alpha$, $\beta\in L$.
    The vacuum vector $\mathbbm 1$ of $V_L$ is $1\otimes e^0$ and the Virasoro element $\omega_L$ is $\frac{1}{2}\sum^d_{i=1} \beta_i(-1)^2\cdot \mathbbm 1$ where $\{\beta_1,\cdots,\beta_d\}$ is an orthonormal basis of $\mathfrak h$.
    For the explicit definition of the corresponding vertex operators, we shall refer to \cite{FLM} for details.

\begin{nota}
Let $L$ be an integral lattice. Its dual lattice $L^*$ is defined to be
\[
L^* =\{\be \in \Q\otimes_\Z L\mid \la \al, \be \ra \in \Z \text{ for all }  \al\in L\}.
\]
The discriminant group of $L$ is the quotient group $\mathcal{D}(L)=L^*/L$. We also denote the isometry group of $L$ by $O(L)$.
\end{nota}

\begin{df}\label{tm}
Let $X$ be a subset of a Euclidean space. Define $t_X$ to be the orthogonal transformation which is $-1$ on $X$ and is $1$ on
$X^\perp$.
\end{df}

\begin{df}\label{rssd}
A sublattice $M$ of an integral lattice $L$ is {\it RSSD (relatively semiselfdual)} if and only if $2L\le M + ann_L(M)$. This implies
that $t_M$ maps $L$ to $L$ and hence $t_M\in O(L)$.

The property that $2 M^* \le M$ is called {\it SSD (semiselfdual)}. It implies the RSSD property.
\end{df}

\begin{rem}
Let $\sqrt{2}E_8$ be the $\sqrt{2}$ times of the root lattice $E_8$. Then $\mathcal{D}(\sqrt{2}E_8)\cong 2^8$. Therefore, the lattice $\sqrt{2}E_8$ is an SSD and thus, for any sublattice $E\cong \sqrt{2}E_8\subset L$, the map $t_E$ defines an
involution on $L$. We call such an involution $t_E$ an SSD involution associated to $E$.
\end{rem}

\begin{df}
Let $A$ and $B$ be integral lattices with the inner products $\langle\ ,\ \rangle_A$ and $\langle\ ,\ \rangle_B$, respectively. The
\emph{tensor product} of the lattice $A$ and $B$ is defined to be the integral lattice which is isomorphic to $A\otimes_\mathbb Z
B$ as a $\mathbb Z$-module and has the inner product given by
\begin{equation}
\langle \alpha\otimes\beta,\alpha'\otimes\beta'\rangle=\langle\alpha,\alpha'\rangle_A\cdot\langle\beta,\beta'\rangle_B,
\end{equation}
for any $\alpha$, $\alpha'\in A$, and $\beta,\beta'\in B$. We simply denote the tensor product of the lattices $A$ and $B$ by
$A\otimes B$.
\end{df}

The next lemma can be found in \cite[Lemma 3.2 and 6.4]{GL} (cf. \cite{GL3c}).
\begin{lem}\label{gl3c}
Let $L$  be a positive definite  even lattice with no roots, i.e., $L(2)=\{\al\in L\mid \la \al, \al\ra =2\}=\emptyset$. Let $M$ and
$N$ be $\sqrt{2}E_8$-sublattices of $L$. Suppose $M\cap N=0$ and $t_Mt_N$ has order 3. Then $M+N\cong
A_2\otimes E_8$.
\end{lem}

\subsection{Ising vectors and Miyamoto involutions}

\begin{df} \label{vvc}
Let $V$ be a VOA. An element $v\in V_2$ is called a \emph{simple Virasoro vector} of central charge $c$
 if the subVOA $\mathrm{Vir}(v)$ generated by $ v$ is isomorphic to the simple Virasoro VOA $L(c,0)$ of central charge $c$.
\end{df}

\begin{df} \label{Ising vector}
A simple Virasoro vector of central charge $1/2$ is called an \emph{Ising vector}.
\end{df}

\begin{rem} \label{3irr module} It is well known the VOA $L(\frac{1}{2},0)$ is rational
and it has exactly 3 irreducible modules $L(\frac{1}{2},0)$, $L(\frac{1}{2},\frac{1}{2})$,
and $L(\frac{1}{2},\frac{1}{16})$ (cf. \cite{dmz,miy}).
\end{rem}

\begin{rem}\label{Veh}
Let $V$ be a VOA and let $e\in V$ be an Ising vector.
Then we have the decomposition
\[
V=V_{e}(0)\oplus V_{e}(\frac{1}{2})\oplus V_{e}(\frac{1}{16}),
\]
where $V_{e}(h)$ denotes the sum of all irreducible $\mathrm{Vir}(e)$-submodules of
$V$ isomorphic to $L(\frac{1}{2},h)$ for $h\in\{0,\frac{1}{2},\frac{1}{16}\}$.
\end{rem}

\begin{thm}[\cite{miy}]\label{taue}
The  linear map $\tau_{e}:\, V\rightarrow V$ defined by
\begin{equation}
\tau_{e}:=\begin{cases}
1 & \text{on \ensuremath{V_{e}(0)\oplus V_{e}(\frac{1}{2})}},\\
-1 & \text{on \ensuremath{V_{e}(\frac{1}{16})}},
\end{cases}\label{eq:taue}
\end{equation}
is an automorphism of $V$.

On the fixed point subspace $
V^{\tau_{e}}=V_{e}(0)\oplus V_{e}(1/2)$,  the  linear map $\sigma_{e}:\, V^{\tau_{e}}\rightarrow V^{\tau_{e}}$ defined
by
\begin{equation}
\sigma_{e}:=\begin{cases}
1 & \text{on \ensuremath{V_{e}(0)}},\\
-1 & \text{on \ensuremath{V_{e}(\frac{1}{2})}},
\end{cases}\label{eq:-1}
\end{equation}
is an automorphism of $V^{\tau_{e}}$.
\end{thm}

\subsubsection{$2B$ and $3C$ algebras}

Next  we recall the properties of the $2B$-algebra $U_{2B}$ and the $3C$-algebra $U_{3C}$ from \cite{LYY2} (see also \cite{Sa}).

\begin{lem}[cf. \cite{LYY2}]
Let $e$ and $f$ be Ising vectors such that $\la e,f\ra=0$. Then the subVOA $VOA(e,f)$ generated by $e$ and $f$ is isomorphic to
$L(\frac{1}2,0)\otimes L(\frac{1}2,0)$. The VOA $VOA(e,f)\cong L(\frac{1}2,0)\otimes L(\frac{1}2,0)$ is called the $2B$-algebra
and is denoted by $U_{2B}$.
\end{lem}

The followings can also  be found in \cite[Section 3.9]{LYY2}.
\begin{lem}\label{eiej}
Let $U=U_{3C}$ be the $3C$-algebra defined as in \cite{LYY2}. Then
\begin{enumerate}

\item  $U_1=0$ and $U$ is generated by its weight 2 subspace $U_2$ as a VOA.

\item $\dim U_2=3$ and it is spanned by 3 Ising vectors.

\item There exist exactly 3 Ising vectors in $U_2$, say, $ e^0,e^1, e^2$. Moreover,
we have
\[
(e^i)_1 (e^j) =\frac{1}{32}( e^i+e^j -e^k), \quad \text{ and } \quad \la e^i, e^j\ra =\frac{1}{2^8}
\]
for $i\neq j$ and $\{i,j,k\}=\{0,1,2\}$.

\item The subgroup generated by $\tau_{e^0}$, $\tau_{e^1}$, and $\tau_{e^2}$ is isomorphic to the symmetry group $S_3$ and $\tau_{e^i}(e^j)=e^k$, where  $\{i,j,k\}=\{0,1,2\}$.

\item The Virasoro element of $U$ is given by $$\frac{32}{33}(e^0+e^1+e^2).$$

\item Let $a=\frac{32}{33}(e^0+e^1+e^2) -e^0$. Then $a$ is a simple Virasoro vector of central charge $21/22$. Moreover, the subVOA generated by $e^0$ and $a$ is isomorphic to $L(\frac{1}2,0)\tensor L(\frac{21}{22},0)$.
\end{enumerate}
\end{lem}

The decomposition of $U_{3C}$ as a sum of irreducible $\mathrm{Vir}(e^0)\otimes  \mathrm{Vir}(a)$-modules is also obtained.

\begin{prop}[Theorem 3.35 of \cite{LYY2}]
As a module  of $L(\frac{1}2,0)\tensor L(\frac{21}{22},0)$,
\[
\begin{split}
  U_{3C}
  & \cong  L(\frac{1}2,0)\tensor L(\frac{21}{22},0)
    \oplus L(\frac{1}{2},0)\tensor L(\frac{21}{22},8)
    \\
  & \quad \oplus  L(\frac{1}2,\frac{1}2)\tensor L(\frac{21}{22},\frac{7}{2})
    \oplus L(\frac{1}2,\frac{1}2)\tensor  L(\frac{21}{22},\frac{45}{2})
    \\
  & \quad  \oplus L(\frac{1}{2},\frac{1}{16})\tensor L(\frac{21}{22},\frac{31}{16})
    \oplus L(\frac{1}{2},\frac{1}{16})\tensor L(\frac{21}{22},\frac{175}{16}).
\end{split}
\]
\end{prop}

\begin{df}\label{2B3Ctype}
Let $V$ be  a moonshine type VOA (i.e., $V=\oplus_{n=0}^{\infty} V_n$, $\dim V_0=1$ and $V_1=0$) and let $I\subset V$ be a set of Ising vectors. We say that $I$ satisfies the $(2B, 3C)$ condition
if for any $e,f\in I$ with $e\neq f$, the subVOA generated by $e$ and $f$ is either isomorphic to $U_{2B}$ or $U_{3C}$.
\end{df}

\subsection{Ising vectors in $V_L^+$ with $L(2)=\emptyset$}
Next we review the construction of certain Ising vectors in the lattice type VOA $V_L^+$ with $L(2)=\emptyset$ \cite{dlmn}.  There are two different types, which we call $A_1$-type and $E_8$-type.

\paragraph{\bf $A_1$-type Ising vectors}
Let $a\in L$ be a norm $4$ vector. Then the elements
\[
\omega^{\pm}(a) := \frac{1}{16} a(-1)^2\cdot \mathbbm{1} \pm  \frac{1}4 (e^a+e^{-a})
\]
are Ising vectors (see \cite{dlmn} or \cite{dmz}).

\paragraph{\bf $E_8$-type Ising vectors} Let $E\cong\sqrt2 E_8$ be a sublattice of $L$. Define
$$e_E:=\frac{1}{16}\omega_E+\frac{1}{32}\sum_{\alpha\in E(4)}e^\alpha\in V_E, $$
where $\omega_E$ denotes the conformal element of the lattice VOA $V_E$. Then $e_E$ is an Ising vector (c.f. \cite{dlmn,GL,LYY}).

Now let $L$ be a positive definite even lattice and let
\[
L^*=\{ \al\in \Q\otimes_\Z L\mid \la \al, \be\ra\in \Z \text{ for all
} \be\in L\} \] be the dual lattice of $L$.  For $x\in L^*$, define
a $\Z$-linear map $\varphi_x: L \to \Z_2$ by
\[
\varphi_x(y)=\la x,y\ra \quad \mod 2.
\]
Clearly the map
\[
\begin{split}
\varphi: L^* & \longrightarrow \mathrm{Hom}_\Z(L, \Z_2)\\
         x &\longmapsto \ \varphi_x
\end{split}
\]
is a surjective group homomorphism and $\ker \varphi=2 L^*$. Hence, we have $\mathrm{Hom}_\Z (L, \Z_2)\cong L^*/2L^*$. For
$\alpha\in L^*$, $\varphi_\alpha$ induces an automorphism of $V_L$ given by
\[
\varphi_\alpha(u\otimes e^\beta) =(-1)^{\la \alpha,\beta\ra} u\otimes e^\beta \quad
\text{ for } u\in M(1)\text{ and } \beta \in L.
\]

Now suppose $L=E\cong \sqrt{2}E_8$. Then $E^*= \frac{1}2 E$ and $E^*/2E^*= \frac{1}2 E/E\cong E_8/ 2E_8$. For $x\in
E^*=\frac{1}2 E$,
$$\varphi_x (e_E)= \frac{1}{16}\omega' + \frac{1}{32}
\sum_{\al\in E(4)} (-1)^{\langle
x,\al\rangle}(e^\al+e^{-\al})$$ is also an Ising vector in $V_E^+$ and we
have $256( = 2^8)$ Ising vectors of this form.

The Ising vectors in the lattice type VOA $V_L^+$ with $L(2)=\emptyset$ has been classified by Shimakura~\cite{ShI}. They are either of $A_1$ or $E_8$-type.

\begin{thm}
Let $L$ be an even (positive definite) lattice with  $L(2)=\emptyset$. Suppose $e\in V_L^+$ is an Ising vector. Then either $e=\omega^\pm (\al)$ for some $\al\in L$ with $\la \al, \al\ra=4$ or
$e=\varphi_x(e_E)$ for some sublattice $E\cong \sqrt{2}E_8$ in $L$ and $x\in \frac{1}2 E$.
\end{thm}

\subsubsection{Automorphism group of $V_L^+$ for even lattice $L$ with no roots}
Next we recall some information about the automorphism group of the lattice type VOA $V_L^+$ from \cite{Sh} (see also \cite{LS}).

Let $L$ be a positive definite even lattice and
\[
1 \longrightarrow \langle\kappa\rangle \longrightarrow \hat{L} \bar{
\longrightarrow} L \longrightarrow 1
\]
  a central extension of $L$ by $\langle\kappa\rangle\cong\Z_2$ such that the commutator map $c_0(\al, \be)={\la \al, \be\ra}
\mod 2$, for $\al,\be\in L$. The following theorem is well-known (cf. \cite{FLM}):

\begin{thm}\label{thm:4.10} For an even lattice $L$, the sequence
\[
1 \longrightarrow \mathrm{Hom}(L, \Z_2) { \longrightarrow}
\aut \hat{L} \overset{\eta}\longrightarrow O(L)\longrightarrow  1
\]
is exact.
\end{thm}

Recall that $\theta$ is the automorphism of $V_L$ defined by
\[
\theta( \al_1(-n_1)\cdots \al_k(-n_k) e^\al)= (-1)^{k}
\al_1(-n_1)\cdots \al_k(-n_k)\theta (e^\al),
\]
where $\theta(a)=a^{-1}\kappa^{\la \bar{a},\bar{a}\ra/2}$ on $\hat{L}$.

\begin{lem}([\cite{Sh}])\label{lem:4.11}
Let $L$ be a positive definite even lattice without roots, i.e.,  the set $L(2)=\{\al\in L\mid \la \al, \al\ra =2\}=\emptyset$. Then the
centralizer $C_{\aut{V_L}}(\theta)$ of $\theta$ in $\aut{V_L}$ is isomorphic to $\aut{\hat{L}}$. In particular,
$C_{\aut{V_L}}(\theta)\cong \aut{\hat{L}}\cong 2^{rank\,L}. O(L).$
\end{lem}

\begin{prop}[Proposition 4.12 of \cite{LS}] \label{taue=tM}
Let $E\cong {\sqrt{2}E_8}$ be a sublattice of $L$ and $ e$ an Ising  vector in $V_L$ of the form $\varphi_x(e_E)$. Then $\tau_e\in
C_{\aut{V_L}}(\theta)$ and the image of $\tau_e$ under the canonical epimorphism
\[
\eta: C_{\aut{V_L}}(\theta)(\cong
\aut \hat{L}) \longrightarrow C_{\aut{V_L}}(\theta)/Hom(\Lambda,\Z_2)\cong O(L)
\]
is equal to $t_E$, the SSD involution associated to $E$.
\end{prop}

\section{Weyl groups and VOA generated by Ising vectors}\label{sec:3}

In this section, we shall construct a certain VOA generated by Ising vectors such that the subgroup generated by the
corresponding Miyamoto involutions is isomorphic to the Weyl group of a root system of type $A$, $D$ or $E$. The main idea is to consider a certain integral lattice which has many sublattices isometric to $\sqrt{2}E_8$.

We use the standard model for the root lattices of type $A_{n-1}$ and $D_n$~\cite{CS}, namely,
\[
\begin{split}
A_{n-1}&=\{ \sum_{i=1}^n a_i\epsilon_i\mid a_i\in \Z, \sum_{i=1}^n a_i=0\},\\
D_n&= \{ \sum_{i=1}^n a_i\epsilon_i\mid a_i\in \Z, \sum_{i=1}^n a_i\equiv 0 \mod 2 \},
\end{split}
\]
where $\{\epsilon_1, \dots, \epsilon_n\}$ is an orthonormal basis of $\R^n$.

The root lattice of type $E_8$ is then given by
\[
E_8=span_\Z\left\{ D_8, \frac{1}2(\epsilon_1 +\dots +\epsilon_8)\right\}.
\]

Recall that  the lattice VOA $V_{E_8}$ is an irreducible level $1$ highest weight representation
of the affine Kac-Moody Lie algebra of type $E_8^{(1)}$ \cite{FLM}. Moreover,
$V_{E_8}\cong L_{E_8}(1,0)$
as a VOA.

\begin{nota}\label{mij}
Let $X= E_8\perp \cdots\perp E_8$ be an orthogonal sum of $n$ copies of $E_8$.
Let $X^{(i)}$ be the $i$-th copy of $E_8$ in $X$ and let
$ \iota_i:E_8\rightarrow X^{(i)}\subset X$
be the natural injection.

For any $i\neq j$, we define
\[
\nu_{i,j}:=\iota_i-\iota_j\quad \text{ and }\quad \mu_{i,j}:=\iota_i+\iota_j.
\]
We also define $d=d_n : = \iota_1 +\cdots + \iota_n$.
\end{nota}

\begin{lem}\label{AD}
Let $\mathcal{A}_{n-1}= span_\Z\{ \nu_{i,j}(\al)\mid \al \in E_8, 1\leq i<j \leq n\}$ and
$$\mathcal{D}_n= span_\Z\{ \nu_{i,j}(\al), \mu_{i,j}(\al)\mid \al \in E_8, 1\leq i<j \leq n\}.$$
Then
$\mathcal{A}_{n-1}\cong A_{n-1}\otimes E_8$ and $\mathcal{D}_{n}\cong D_n\otimes E_8$.
\end{lem}

\begin{proof}
Let $\{\al_1, \dots, \al_8\}$ be a basis of $E_8$. Then
$\{(\epsilon_i- \epsilon_{i+1})\otimes \al_k\mid i=1, \dots, n-1, k=1, \dots,8\}$
is a basis of $A_{n-1}\otimes E_8$.  Define a $\Z$-linear map $f:A_{n-1}\otimes E_8 \to \mathcal{A}_{n-1}$ by
\[
f( (\epsilon_i- \epsilon_{i+1})\otimes \al_k) = \nu_{i, i+1}(\al_k).
\]
By a direct calculation, it is easy to verify that the Gram matrix
\[
\left (\la  (\epsilon_i- \epsilon_{i+1})\otimes \al_k, (\epsilon_j- \epsilon_{j+1})\otimes \al_\ell\ra \right)
\]
of $A_{n-1}\otimes E_8$ is the same as the Gram matrix
$
\left (\la  \nu_{i, i+1}( \al_k),   \nu_{j, j+1}( \al_\ell) \ra \right)
$
of $\mathcal{A}_{n-1}$.  Therefore, $f$ is an isometry and $\mathcal{A}_{n-1}\cong A_{n-1}\otimes E_8$.
Similarly, we also have $\mathcal{D}_{n}\cong D_{n}\otimes E_8$.
\end{proof}

\medskip

The next two lemmas can also be proved by computing their Gram matrices.

\begin{lem}\label{E8E8}
Let $n=8$ and let
\[
\mathcal{E}_8 =span_\Z\left\{ \mathcal{D}_8, \frac{1}2 d(E_8)\right\} \subset \frac{1}2 X.
\]
Then $\mathcal{E}_8 \cong E_8\otimes E_8$, which is an even unimodular lattice.
\end{lem}

\begin{lem}\label{E6E7}
Let
\[
\begin{split}
\mathcal{E}_7 &= \{b \in  \mathcal{E}_8\mid \la b, \frac{1}2 d(\al)\ra =0 \text{ for all } \al\in E_8\},\\
\mathcal{E}_6 &= \{b \in  \mathcal{E}_8\mid \la b, \frac{1}2 d(\al)\ra =\la b, \mu_{1, 8} (\al)\ra =0 \text{ for all } \al\in E_8\}
\end{split}
\]
Then $\mathcal{E}_7 \cong E_7\otimes E_8$ and $\mathcal{E}_6 \cong E_6\otimes E_8$.
\end{lem}

\medskip

Next we fix a $2$-cocycle on the rational lattice $\frac{1}2 X$.
Let  $\{\al_1, \dots, \al_8\}$ be a basis of the root lattice $E_8$ and set
\[
x_1=\frac{\al_1}2, \ x_2=\frac{\al_2}2,\  \dots,\  x_n=\frac{\al_n}2.
\]
Then $\{x_1, \dots, x_n\}$ is a basis of $\frac{1}2 E_8$.

Notice that  $ 4 \langle x,y \rangle \in \mathbb{Z}$ for $x,y
\in \frac{1}2 E_8$.
By \cite[(5.2.14)]{FLM} and ~\cite[Remark 6.4.12]{LL}, we can define a bilinear 2-cocycle $
\varepsilon_0 : \frac{1}2 E_8 \times \frac{1}2 E_8  \to \mathbb{Z}_8  $ by defining
\begin{align}\label{epsilon0}
 \varepsilon_0( x_i, x_j) =
\begin{cases}
 4 \la x_i, x_j\ra &\mod 8 \quad \text{ if } i >j, \\
 2\la x_i, x_i\ra=1     &\mod 8 \quad \text{ if  } i =j, \\
0 &\mod 8 \quad \text{  if  } i <  j.
 \end{cases}
\end{align}
We also extend the $2$-cocycle $\varepsilon_0$ to
$(\frac{1}2E_8)^n$ by defining
$$\varepsilon_0\big((a_1,\cdots,a_n),(b_1,\cdots,b_n)\big)=\sum^n_{i=1}\varepsilon_0(a_i,b_i)$$
for $a_i,b_i\in \frac{1}2 E_8$, $i=1,\cdots,n$.

In addition, we define $\varepsilon : \frac{1}2 X \times \frac{1}2 X  \to \C^*$ by setting
\[
\varepsilon( \al, \be) := \exp( 2 \pi \sqrt{-1} \varepsilon_0( \al, \be)/8 )\quad \text{ for } \al, \be \in \frac{1}2X.
\]

\begin{lem}  Let $R$ be a  root lattice of type  $A$, $D$, or $E$.
Using the identification of $R\otimes E_8$ in Lemma \ref{AD} and Lemma \ref{E8E8},  we have
\[
\varepsilon(\al\otimes \gamma, \be \otimes \gamma) = -1
\]
for any $\al, \be \in \Phi(R)$ with $\la \al, \be\ra =\pm 1$.
\end{lem}

\begin{proof}
We continue to use the standard model for root lattices of type $A$, $D$, and $E$.
For $R=A_{n-1}$ or $D_n$,  $\la \al, \be\ra =\pm 1$ if and only if  $\al=\epsilon_i\pm \epsilon_j$ and $\be=\pm (\epsilon_j\pm \epsilon_k)$ for some $i,j,k$. Thus, by the identification in Lemma \ref{AD}, we have
\[
\varepsilon(\al\otimes \gamma, \be \otimes \gamma) = \varepsilon( \pm  \iota_j(\gamma), \pm  \iota_j(\gamma) ) =-1.
\]

For $R=E_8, E_7,E_6$, let $\al=\sum_{i=1}^8 a_i \epsilon_i$ and $\be=\sum_{i=1}^8 b_i \epsilon_i$ for $a_i, b_i\in \frac{1}2\Z$. Then
 \[
\varepsilon_0(\al\otimes \gamma, \be \otimes \gamma) =\sum_{i=1}^8 a_ib_i \varepsilon_0( \pm  \iota_i(\gamma), \pm  \iota_i(\gamma) ) = 4 \la \al, \be\ra =\pm 4 .
\]
Hence, $ \varepsilon(\al\otimes \gamma, \be \otimes \gamma)=-1$.
\end{proof}

\begin{nota}
Let $R$ be a root lattice. We use $\Phi(R)$ and $\Phi^+(R)$ to denote the set of all roots and positive roots of $R$, respectively.
\end{nota}

\begin{nota}
For any $\al\in \Phi(R)$, set
\[
M_\al :=\Z\al\otimes E_8 \subset R\otimes E_8.
\]
Then $M_\al\cong \sqrt{2}E_8$. We also define
\[
e(\al) :=e_{M_\al}=\frac{1}{16}\omega_{M_\al} + \frac{1}{32} \sum_{\be\in \Phi^+(E_8)} (e^{\al\otimes \be }+ e^{-\al\otimes \be })
\]
to be the standard Ising vector associated to $M_\al$.  Note that $M_{\al}=M_{-\al}$ and $e(\al) =e(-\al)$.
\end{nota}

\begin{rem}
Again using the identification in Lemmas \ref{AD} and \ref{E8E8}, the 2-cocyle $\varepsilon_0$ defined in Equation \eqref{epsilon0} is trivial on $M_\al$ for any root $\al \in \Phi(R)$
\end{rem}

\begin{lem}
For any $\al, \be\in \Phi^+(R)$ with $\al\neq \be$,  we have
\[
M_\al+M_\be \cong
\begin{cases}
A_2\otimes E_8, & \text{ if } \la \al, \be \ra =\pm 1, \\
\sqrt{2}E_8 \perp \sqrt{2}E_8, & \text{ otherwise } .
\end{cases}
\]
\end{lem}

\begin{proof}
Since $\al\neq\be $ and $\al, \be$ are roots, we have $\la \al, \be\ra =\pm 1$ or $0$.

If $\la \al, \be \ra =\pm 1$, then $M_\al\cap M_\be=0$ and $t_{M_\al}t_{M_\be}$ has order 3. Hence $M_\al+M_\be \cong
A_2\otimes E_8$  by Lemma \ref{gl3c}.

If $\la \al, \be \ra =0$, then $\la M_\al, M_\be\ra =0 $ and we have $M_\al+M_\be\cong \sqrt{2}E_8 \perp \sqrt{2}E_8 $ as
desired.
\end{proof}

\medskip

\begin{lem}\label{eprod}
Let $\al, \beta$ be roots of $R$. Then we have
\[
e(\al)_1 e(\be)  = \frac{1}{32}\left(e(\al)+ e(\be) - e(\al+\be)\right)
\]
if $\la \al, \be\ra =-1$.
\end{lem}

\begin{proof}
Recall from \cite[(8.9.55)]{FLM} that
\[
\left(h(-1)^2\cdot\mathbbm{1}\right)_1 \left(g(-1)^2\cdot\mathbbm{1}\right)
= 4\la h,g\ra h(-1) g(-1)\cdot\mathbbm{1}
\]
and
\[
(e^a+e^{-a})_1 (e^{b}+e^{-b}) =
\begin{cases}
\varepsilon(a, b) (e^{a\pm b} + e^{-(a\pm b)})& \text{ if } \la a, b\ra =\mp 2, \\
a(-1)^2\cdot \mathbbm{1} &\text{ if } \la a, b\ra =\pm 4,\\
0 & \text{ otherwise},
\end{cases}
\]
for any norm $4$ vectors $a,b$ and for any $g,h\in \mathfrak{h}$.

Let $\{h_1, \dots, h_8\}$ be an orthonormal basis of $\Q\otimes M_\al$ and  let $h_i'=t_{M_\al}t_{M_\be} h_i$. Then
$\{h_1', \dots, h_8'\}$ is an orthonormal basis of $\Q\otimes M_\be$ and we have
\[
\la h_i, h_j' \ra =-\frac{1}2\la h_i,h_j\ra =-\frac{1}2 \delta_{i,j}.
\]
In addition, $\omega_{M_\al} = \frac{1}2 \sum h_i(-1)^2\cdot \mathbbm{1}$,  $\omega_{M_\be} = \frac{1}2 \sum h_i'(-1)^2\cdot \mathbbm{1}$ and
\[
(\omega_{M_\al})_1(\omega_{M_\be}) = \frac{1}4 \sum_{i=1}^8 4\cdot (-\frac{1}2) h_i(-1)h_i'(-1)\cdot \mathbbm{1}
= - \frac{1}2 \sum_{i=1}^8 h_i(-1)h_i'(-1)\cdot \mathbbm{1}.
\]
Thus,
\[
\begin{split}
e(\al)_1e(\be) &= \frac{1}{2^{10}} \left( 2\omega_{M_\al} +\sum_{\gamma\in \Phi^+(E_8)} (e^{\al\otimes \gamma} + e^{-\al\otimes \gamma} )\right)_1
\left( 2\omega_{M_\be} +\sum_{\gamma\in \Phi^+(E_8)} (e^{\be\otimes \gamma} + e^{-\be\otimes \gamma} )\right)\\
& = \frac{1}{2^{10}} \left[ - 2\sum_{i=1}^8 h_i(-1)h_i'(-1)\cdot \mathbbm{1} + \sum_{\gamma\in \Phi^+(E_8)} (e^{\al\otimes \gamma} + e^{-\al\otimes \gamma} ) + \sum_{\gamma\in \Phi^+(E_8)} (e^{\be\otimes \gamma} + e^{-\be\otimes \gamma} )\right.\\
& \quad  \left. + \left(\sum_{\gamma\in \Phi^+(E_8)} (e^{\al\otimes \gamma} + e^{-\al\otimes \gamma} )\right)_1
\left( \sum_{\gamma\in \Phi^+(E_8)} (e^{\be\otimes \gamma} + e^{-\be\otimes \gamma} )\right) \right ]\\
& = \frac{1}{2^{10}} \left[ - 2\sum_{i=1}^8 h_i(-1)h_i'(-1)\cdot \mathbbm{1} + \sum_{\gamma\in \Phi^+(E_8)} (e^{\al\otimes \gamma} + e^{-\al\otimes \gamma} ) + \sum_{\gamma\in \Phi^+(E_8)} (e^{\be\otimes \gamma} + e^{-\be\otimes \gamma} )\right.\\
& \quad  \left. + \sum_{\gamma\in \Phi^+(E_8)} \varepsilon(
\al\otimes \gamma, \be\otimes \gamma) (e^{(\al+\be)\otimes \gamma} + e^{-(\al+\be)\otimes \gamma} ) \right ]\\
&= \frac{1}{32}\left(e(\al)+ e(\be) - e(\al+\be)\right).
\end{split}
\]
Note that $-2h_i(-1)h_i'(-1)\cdot \mathbbm{1} =[  h_i(-1)^2 + h_i'(-1)^2 -(h_i+h_i')(-1)^2 ]\cdot \mathbbm{1}.$
\end{proof}


\medskip

    \begin{lem}\label{inner}
        For any $\al, \be \in \Phi^+(R)$, we have
        \begin{eqnarray*}
            \langle e(\al), e(\be)\rangle=\begin{cases}
                \frac{1}{4}&\text{ if } \al=\be,\\
                \frac{1}{2^8}&\textrm{ if } \la \al, \be \ra =\pm 1,\\
                0&\textrm{ if } \la \al, \be \ra =0.\\
            \end{cases}
        \end{eqnarray*}
    \end{lem}

    \begin{proof}
   First we note that $\la \al, \be \ra =0, \pm 1$, or $2$ for any $\al, \be\in \Phi^+(R)$.

If $\la \al, \be \ra =2$, then $\al=\be$ and $e(\al)=e(\be)$. Thus, $ \langle e(\al), e(\be)\rangle=1/4$.

If $\la \al, \be \ra =0$, then $M_\al\perp M_\be$ and hence $\la e(\al), e(\be) \ra =0$.

If $\la \al, \be \ra =\pm 1$, then $M_\al\cap M_\be =0$ and $M_\al+M_\be\cong A_2\otimes E_8$. In this case, we have
\[
 \langle e(\al), e(\be)\rangle= \left(\frac{1}{16}\right)^2 \la \omega_{M_\al} , \omega_{M_\be}\ra.
\]
As in the proof Lemma \ref{eprod}, let $\{h_1, \dots, h_8\}$ be an orthonormal basis of $\Q\otimes M_\al$ and  let $h_i'=t_{M_\al}t_{M_\be} h_i$. Then we have
$\la h_i, h_j' \ra =-\frac{1}2 \delta_{i,j}$, $\omega_{M_\al} = \frac{1}2 \sum h_i(-1)^2\cdot \mathbbm{1}$ and   $\omega_{M_\be} = \frac{1}2 \sum h_i'(-1)^2\cdot \mathbbm{1}$. Thus,
\[
\la \omega_{M_\al}, \omega_{M_\be}\ra = \frac{1}4 \sum_{i=1}^8 2\la h_i, h_i'\ra ^2 = \frac{1}4 \cdot 8\cdot 2\cdot \left(-\frac{1}{2}\right)^2=1.
\]
Hence, $\langle e(\al), e(\be)\rangle= 1/{2^8}$ as desired.
    \end{proof}

    \begin{cor}\label{2b3ccon}
       Let $\al, \be \in \Phi^+(R)$ with $\al \neq \be$. Then  the subVOA  generated by $e(\al)$ and $e(\be)$ is
        \[
          \mathrm{VOA}(e(\al),e(\be))
\cong
            \begin{cases}
                U_{3C}&\textrm{ if }  \la \al, \be \ra =\pm 1,  \\
                U_{2B}&\textrm{ if }  \la \al, \be \ra =0.
            \end{cases}
        \]
    \end{cor}

    \begin{nota}
        From now on, we define $\mathcal{W}_R$ to be the sub-VOA generated by $I=I_R=\{e(\al) \mid \al\in \Phi^+(R)\}$.
        We also denote the subgroup generated by the Miyamoto involutions $\{\tau_e|\ e\in I\}$ by $G(R)$ or simply by $G$.
    \end{nota}

\begin{rem}
By definition, it is clear that $e(\al) \in V_{R\otimes E_8}^+$ for all $\al\in \Phi^+(R)$.  Thus, $\mathcal{W}_R\subset V_{R\otimes
E_8}^+ $ has zero weight one subspace. Hence $\mathcal{W}_R$ is a moonshine type VOA. Note also that the set
$I=I_R=\{e(\al)\mid \al \in \Phi^+(R)\}$ satisfies the $(2B, 3C)$ condition by Corollary \ref{2b3ccon}.
\end{rem}

\subsubsection{Groups generated by Miyamoto involutions}
For any $g\in O(R)$ and $g'\in O(E_8)$, the map $g\otimes g'$ acts as an isometry on $R\otimes E_8$ by
\[
g\otimes g'(\al\otimes \be)= g(\al)\otimes g'(\be) \quad \text{ for } \al\in R, \be \in E_8.
\]
Therefore,  we can view the central product $O(R)* O(E_8)$ as a subgroup of $O(R\otimes E_8)$.

\medskip

\begin{nota}
For $\al\in \Phi(R)$, let $r_\al\in Weyl(R)$ be the reflection associated to $\al$, i.e.,
\[
r_\al(\be) = \be - \la \al, \be \ra \al \quad \text{ for } \be \in R.
\]
\end{nota}

\begin{lem}
Let $\al\in \Phi(R)$. Then  the SSD involution $t_{M_\al}$  acts as  $r_\al \otimes 1 $ on $R\otimes E_8$.
\end{lem}

\begin{proof}
Recall that $t_{M_\al}$ acts as $-1$ on $M_\al$ and acts as $1$ on $ann_{R\otimes E_8}(M_\al)$. Therefore,
\[
t_{M_\al}(\al\otimes \gamma) = -\al\otimes \gamma \quad  \text{ and } \quad
t_{M_\al}(\be \otimes \gamma) = \be \otimes \gamma
\]
for any $\gamma \in E_8$ and $\be \in R$ with $\la \al, \be \ra =0$.   Hence $$t_{M_\al}(\be \otimes \gamma) = r_\al(\be) \otimes \gamma\quad \text{ for any } \gamma \in E_8, \be \in R$$
and we have $t_{M_\al}= r_\al \otimes 1 $ on $R\otimes E_8$.
\end{proof}

\begin{prop}\label{imageeta}
Let $\bar{\eta}: C_{\aut(V_{R\otimes E_8})}(\theta) /\la \theta\ra  \longrightarrow  O( R\otimes E_8)/ \la \pm 1\ra $ be the natural
epimorphism defined in Theorem \ref{thm:4.10}. Then
\[
\bar{\eta}(G(R)) \cong Weyl(R)/ (Weyl(R)\cap \la \pm 1\ra ).
\]
\end{prop}

\begin{proof}
Since $t_{M_\al} =r_\al\otimes 1$ on $R\otimes E_8$ and $\eta(\tau_{e(\al)} ) = t_{M_\al}$, we have
\[
\eta(G(R))=\la t_{M_\al}\mid \al\in \Phi^+(R)\ra = \la r_\al\otimes 1\mid \al\in \Phi^+(R)\ra \cong Weyl(R).
\]
Therefore, we have
\[
\bar{\eta}(G(R)) \cong Weyl(R)/ (Weyl(R)\cap \la \pm 1\ra )
\]
as desired.
\end{proof}

\begin{lem}\label{transeij}
For any   $\al, \be \in \Phi(R)$, we have
\[
\tau_{e(\al) }(e(\be) ) =
\begin{cases}
e(\al\pm \be) & \text{ if } \la \al, \be \ra = \mp 1, \\
e(\be) & \text{ otherwise.}
\end{cases}
\]
 \end{lem}

 \begin{proof}
 It follows from Lemma \ref{eprod} and Lemma \ref{eiej}~(4).
 \end{proof}

    \begin{thm}\label{GR}
        Let $G=G(R)$ be the subgroup  generated by $\{\tau_{e(\al)}| \al \in \Phi^+(R)\}$.  Then
       \[
       G(R)\cong Weyl(R)/ (Weyl(R)\cap \la \pm 1\ra )
       \]
    as a subgroup of $\aut(\mathcal{W}_R)$.
    \end{thm}

\begin{proof}
By Lemma \ref{transeij}, $G$ acts as a permutation group on $\{ e(\al)\mid \al\in \Phi^+(R)\}$.

Let $g\in \ker \bar{\eta} \cap G$. Then $g(e(\al))=e(\be)$ for some $\be \in \Phi^+(R)$.

Since $g\in \ker \bar{\eta}$, we also have $g=\varphi_v$ for some $v\in R\otimes E_8$ by Theorem \ref{thm:4.10}.
Thus, $e(\be)= \varphi_v(e(\al))$ and we must have $e(\be)=e(\al)$.  Therefore, $g$ acts as an identity on $\W_R$ since $\W_R$ is generated by $\{ e(\al)\mid \al\in \Phi^+(R)\}$.  The desired conclusion now follows from Proposition \ref{imageeta}.
\end{proof}

\section{Griess algebra and the central charge of $\W_R$}   Next we study the Griess algebra and determine the central charge of $\W_R$.

\begin{nota}
Let $A$ be a set. We use $\#(A)$ to denote the number of elements in $A$.
\end{nota}

\begin{lem}
Let $\mathcal{G}_R$  be the Griess subalgebra of $\W_R$ generated by $\{e(\al) \mid \al\in \Phi^+(R)\}$. Then $\mathcal{G}_R=span_\Z\{e(\al) \mid \al\in \Phi^+(R)\}$ and $\dim \mathcal{G}_R =\#(\Phi^+(R))$.
\end{lem}

\begin{proof}
By Lemma \ref{eprod}, it is clear that $span_\Z\{e(\al) \mid \al\in \Phi^+(R)\}$ is closed under the Griess algebra product. Therefore, we have $\mathcal{G}_R=span_\Z\{e(\al) \mid \al\in \Phi(R)\}$.

That $\{e(\al) \mid \al\in \Phi^+(R)\}$ is linearly independent follows from the fact that $M_\al\cap M_\be =0$ whenever $\al
\neq \pm \be$. Note also that
$$e(\al)=\frac{1}{16}\omega_{M_\al} + \frac{1}{32} \sum_{\be\in \Phi^+(E_8)} (e^{\al\otimes \be }+ e^{-\al\otimes \be })$$
and $\{ e^x \mid x\in L\}$ is a basis of the twisted group algebra $\C\{L\}$.
\end{proof}

\begin{lem}\label{mal}
Let $R$ be a simple root lattice of type A,D, or E.
For $\alpha\in R$, we define
\[
\Phi^+_{\alpha}(R)=\{ \beta\in \Phi^+(R)\mid  \la \alpha, \beta\ra =\pm 1\}\quad \text{ and } \quad
m_{\alpha}=m_\al(R) = \#(\Phi^+_{\alpha}(R)).
\]
Then $m_\al = 2(h-2)$, where $h$ is the Coxeter number of $R$.
\end{lem}

\begin{proof}
The lemma can be proved by case by case checking. First we note that $$\#(\Phi(R))=h\cdot \ell,$$ where $\ell$ is the rank of $R$. Moreover, $\la \al, \be \ra = 0, \pm 1, \pm 2$ for any $\al, \be \in \Phi(R).$

Set $P(\al) = \{ \be \in \Phi(R)\mid \la \al, \be \ra =0\}$. Then
\[
m_\al=m_\al(R)= \frac{1}2 \left( \#(\Phi(R)) - \#(P(\al)) -2\right).
\]

For $R=A_{n-1}$, $h=n$ and $P(\al)$ is a root system of type $A_{n-3}$. Thus,
$$m_\al= \frac{1}2 ( n(n-1)- (n-2)(n-3) -2)= 2(n-2).$$

For $R=D_{n}$, $h=2n-2$ and $P(\al)$ is of type $A_1+ D_{n-2}$. Thus,
$$ m_\al= \frac{1}2 ( 2n(n-1)- 2(n-2)(n-3)-2 -2)= 2(2n-2-2).$$

For $R=E_6$, $ h=12$, $P(\al)$ is of type $A_5$ and $m_\al=\frac{1}2(72-30-2)=20=2(12-2)$.

For $R=E_7$, $h=18$, $P(\al)$ is of type $D_6$ and $m_\al=\frac{1}2(126-60-2)=32=2(18-2)$.

For $R=E_8$, $h=30$, $P(\al)$ is of type $E_7$ and $m_\al=\frac{1}2(240-126-2)=56=2(30-2)$.
\end{proof}

 \begin{lem}
 Let $h$ be the Coxeter number of $R$  and let $\ell$ be the rank of $R$. Define
 \[
 w_R=\frac{32}{h+30}\sum_{\al\in \Phi^+(R)} e(\al).
 \]
 Then $w_R$ is a Virasoro vector of central charge
$(8\ell h)/(h+30)$. Moreover,  we have
 \[
 (w_R)_1(e(\al))= (e(\al))_1 (w_R)=2 e(\al)  \qquad \text{ for all }  \al \in \Phi^+(R).
 \]
 In particular,  $w_R$ is the conformal element of $\W_R$.
 \end{lem}
\begin{proof}
By Lemma \ref{eiej} (3), we have
\[
(e(\al))_1e(\be)=
\begin{cases}
2e(\al) & \text{ if }  \al=\pm \be,\\
\frac{1}{32} (e(\al) +e(\be) - e(\al \mp  \be) )& \text{ if } \la \al, \be \ra =\pm 1,\\
0 & \text{ if }  \la \al, \be \ra =0.
\end{cases}
\]
Hence
\[
\left(\sum_{\al \in \Phi^+(R)} e(\al) \right)_1 \sum_{\be \in \Phi^+(R) } e(\be)  = \sum_{\al \in \Phi^+(R)}  \left( 2e(\al)+ \frac{1}{32}m_\al e(\al)\right),
\]
where $m_\al =\#( \{ \beta\in \Phi^+(R)\mid  \la \alpha, \beta\ra =\pm 1\}) =2(h-2)$ (cf. Lemma \ref{mal}). Therefore,
\[
\left(\sum_{\al \in \Phi^+(R)} e(\al) \right)_1 \sum_{\al \in \Phi^+(R) } e(\al)   = (2+2(h-2)/32) \sum_{\al \in \Phi^+(R)} e(\al)=
\frac{(h+30)}{16} \sum_{\al \in \Phi^+(R)} e(\al)
\]
and $(w_R)_1(w_R)=2w_R$. Moreover,
\[
\la w_R,w_R\ra = \left( \frac{32}{h+30}\right) ^2\sum_{\al\in \Phi^+(R)}  \left( \frac{1}4 +\frac{1}{2^8}\cdot 2(h-2)\right) = \frac{4h\ell}{h+30}.
\]
Hence, the central charge of $w$ is $8h \ell/(h+30)$. In addition, we have
\[
(w_R)_1 e(\al) = \frac{32}{h+30}\left(\sum_{\be\in \Phi^+(R)} e(\be)\right)_1  e(\al) = \frac{32}{h+30}  \left( 2e(\al)+ \frac{1}{32}\cdot 2(h-2) e(\al)\right)= 2e(\al)
\]
as desired.
\end{proof}

\medskip

\section{More about $\W_{A_{n-1}}$}

In this section, we shall give more details about the VOA $\W_{A_{n-1}}$. We use the same notation as in Section \ref{sec:3}.  In particular, $X=E_8^n$ and
\[
\begin{split}
\mathcal{A}_{n-1} = span_\Z\{ \nu_{i,j}(\al)\mid \al \in E_8, 1\leq i<j \leq n\}\cong A_{n-1}\otimes E_8. 
\end{split}
\]
For simplicity, we denote  $$M_{i,j} =M_{\epsilon_i-\epsilon_j} = \{\nu_{i,j} (\al)\mid \al\in E_8\}$$ and
\[
e^{i,j} = e_{M_{i,j}}= \frac{1}{16}\omega_{M_{i,j}} + \frac{1}{32} \sum_{\be\in \Phi^+(E_8)} (e^{\nu_{i,j}(\be) }+ e^{-\nu_{i,j}(\be)}).
\]

\medskip

For  $R=A_{n-1}$, the subgroup generated by Miyamoto involutions $\{\tau_{e^{i,j}}\mid 1\leq i<j\leq n\}$ can be described more explicitly.

\begin{lem}\label{tauM}
Let $M=\{ (\al, -\al)\mid \al\in E_8\}\subset E_8\perp E_8$. Then $\tau_{e_M}$ acts as a transposition on the two tensor copies
of the lattice VOA $V_{E_8^2}\cong V_{E_8}\otimes V_{E_8}$.
\end{lem}

\begin{proof}
Set $\mathfrak{h}=\C \otimes_{\Z} (E_8\perp E_8)$.  First we recall from \cite{LS} (see also \cite{GL}) that $\tau_{e_M}$ acts on
$\mathfrak{h}(-1)\cdot 1 (\subset (V_{E_8^2})_1) $ as an SSD involution $t_M$, i.e.,  it acts as $-1$ on  $\C\otimes_\Z M $ and
acts as an identity  on the orthogonal complement.

For any $x,y\in E_8$, we have
\[
(x,y)= \frac{1}2[ (x-y, y-x)+ (x+y, x+y)].
\]
Hence, $t_M(x,y) = \frac{1}2[ -(x-y, y-x)+ (x+y, x+y)]=(y,x)$ and
\[
\tau_{e_M}(x,y)(-1)\cdot \mathbbm{1} = (y,x)(-1)\cdot \mathbbm{1}.
\]
Moreover, since $\varepsilon_0(x, -\frac{1}2 x)\equiv -\la x,x\ra \equiv -2 \mod 8$ for any root $x\in E_8$, we have
\[
\begin{split}
&\ (e_M)_1 ( e^{\frac{1}2(x,-x)} + e^{\frac{1}2(-x,x)})\\
 =&\  \frac{1}{16} \cdot \frac{1}2 ( e^{\frac{1}2(x,-x)} + e^{\frac{1}2(-x,x)})  + \frac{1}{32} \cdot \varepsilon(x, -\frac{1}2x)^2
( e^{\frac{1}2(x,-x)} + e^{\frac{1}2(-x,x)})\\
=&\  \frac{1}{16} \cdot \frac{1}2 ( e^{\frac{1}2(x,-x)} + e^{\frac{1}2(-x,x)})  - \frac{1}{32}
( e^{\frac{1}2(x,-x)} + e^{\frac{1}2(-x,x)})=0
\end{split}
\]
and
\[
\begin{split}
&\ (e_M)_1 ( e^{\frac{1}2(x,-x)} - e^{\frac{1}2(-x,x)})\\
 =&\  \frac{1}{16} \cdot \frac{1}2 ( e^{\frac{1}2(x,-x)} - e^{\frac{1}2(-x,x)})  - \frac{1}{32} \cdot \varepsilon(x, -\frac{1}2x)^2
( e^{\frac{1}2(x,-x)} - e^{\frac{1}2(-x,x)})\\
=&\  \frac{1}{16} \cdot \frac{1}2 ( e^{\frac{1}2(x,-x)} - e^{\frac{1}2(-x,x)}) + \frac{1}{32}
( e^{\frac{1}2(x,-x)} - e^{\frac{1}2(-x,x)})\\
= & \  \frac{1}{16}  ( e^{\frac{1}2(x,-x)} - e^{\frac{1}2(-x,x)}).
\end{split}
\]
In other words, $( e^{\frac{1}2(x,-x)} + e^{\frac{1}2(-x,x)})$ is a highest weight vector of $e_M$ of weight $0$ and $(
e^{\frac{1}2(x,-x)} - e^{\frac{1}2(-x,x)})$ is a highest weight vector of $e_M$ of weight $1/16$ for any root $x\in E_8$.

Therefore, we have
\[
\begin{split}
\tau_{e_M}( e^{\frac{1}2(x,-x)} + e^{\frac{1}2(-x,x)}) &= e^{\frac{1}2(x,-x)} + e^{\frac{1}2(-x,x)},\qquad \text{ and }\\
\tau_{e_M}( e^{\frac{1}2(x,-x)} - e^{\frac{1}2(-x,x)}) &= -(e^{\frac{1}2(x,-x)} - e^{\frac{1}2(-x,x)}).
\end{split}
\]
Note that $e^{(x,0)}=  e^{\frac{1}2(x,x)}_{-1} e^{\frac{1}2(x,-x)}$   and $e^{(0, x)}=  e^{\frac{1}2(x,x)}_{-1} e^{\frac{1}2(-x,x)}$.
Moreover,
\[
\begin{split}
e^{\frac{1}2(x,-x)} = & \frac{1}2\left[ ( e^{\frac{1}2(x,-x)} +
e^{\frac{1}2(-x,x)}) + ( e^{\frac{1}2(x,-x)} - e^{\frac{1}2(-x,x)})\right],\\
e^{\frac{1}2(-x,x)} =& \frac{1}2\left[ ( e^{\frac{1}2(x,-x)} +
e^{\frac{1}2(-x,x)}) - ( e^{\frac{1}2(x,-x)} - e^{\frac{1}2(-x,x)})\right].
\end{split}
\]

Thus, we have
\[
\tau_{e_M}(e^{(x,0)})= e^{(0,x)} \quad \text{ and }\quad  \tau_{e_M}(e^{(0,x)})= e^{(x,0)}.
\]
Since $V_{E_8^2}$ is generated by   $e^{(x,0)}$, $e^{(0,x)}$, and $\mathfrak{h}(-1)\cdot \mathbbm{1}$,  we have the desired
conclusion.
\end{proof}

\begin{rem}
Let $N=\{ (\al, \al)\mid \al\in E_8\}\subset E_8\perp E_8$. Then by the same argument as in Lemma \ref{tauM}, we also have
\[
\begin{split}
\tau_{e_N}(x,y)(-1)\cdot \mathbbm{1} &= -(y,x)(-1)\cdot \mathbbm{1}, \\
\tau_{e_N}(e^{(x,y)}) &= e^{-(y,x)}.
\end{split}
\]
\end{rem}

As a corollary, we  have the following.

   \begin{cor}\label{transij}
        For $1\leq i < j\leq n$, $\tau_{e^{i,j}}$ acts as a transposition $(i,j)$ on $V_{E_8}\otimes \cdots \otimes V_{E_8}$, i.e.,
        \begin{align*}
            \begin{array}{ccccccccc}
                \tau_{e^{i,j}}(a_1\cdots&\otimes a_i \otimes& \cdots&\otimes a_j\otimes&\cdots a_n) =a_1\otimes \cdots& \otimes a_j\otimes&\cdots&\otimes a_i
                \otimes& \cdots a_n.\\
                &\uparrow& &\uparrow& &\uparrow& &\uparrow&\\
                &i\textrm{-th}& &j\textrm{-th}& & i\textrm{-th}& &j\textrm{-th}&
            \end{array}
        \end{align*}
    \end{cor}

\begin{cor}\label{Sn}
Let $G=G(A_{n-1})$ be the subgroup generated by $\{\tau_{e^{i,j}}\mid 1\leq i<j\leq n\}$. Then $G$ is isomorphic to the symmetric group $S_n$ of degree $n$.
\end{cor}

\begin{nota}
For any root $\al\in E_8$, denote
\[
\begin{split}
H_\al &= d(\al)(-1)\cdot \mathbbm{1}= (\iota_1(\al)+\dots +\iota_n( \al)(-1) \cdot \mathbbm{1}, \\
E_\al &= e^{\iota_1(\al)} + \dots  +e^{\iota_n(\al)}.
\end{split}
\]
Then $\{H_\al, E_\al\mid \al\in E_8\text{ is a root}\}$ generates a subVOA isomorphic to the affine VOA $L_{E_8}(n,0)$ in $V_X$
associated to the affine Kac Moody Lie algebra of type $E_8^{(1)}$ (see \cite[Proposition 13.1]{DL} and \cite{FZ}).
\end{nota}

\begin{lem}\label{lem:5.2}
Let $\tilde{W}= \Com_{V_L}(L_{E_8}(n,0))$. Then all  $e^{i,j}, 1\leq i<j\leq n,$ are contained in $\tilde{W}$.  Moreover,
$\W_{A_{n-1}} =VOA(e^{i,j}\mid  1\leq i<j\leq n)$ is a full subVOA of $\tilde{W}$.
\end{lem}

\begin{proof}
Recall that $H_\al= (\iota_1+\cdots+\iota_n)(\al)(-1)\cdot \mathbbm{1}$ and $E_\al = e^{\iota_1(\al)} + \cdots +e^{\iota_n(\al)}$
for any root $\al \in E_8$. Since $M_{i,j} \perp  \{(\iota_1+\cdots+\iota_n)(\al)\mid \al\in E_8\}$, it is clear that
\[
(H_\al)_k e^{i,j}=0 \text{ for any } k\geq 0.
\]
Moreover,
\[
(E_\al)_k e^{i,j} =0  \text{ for any } k\geq 1,
\]
since  $(V_{\mathcal{A}_{n-1}}^+)_1=0$.

By a direct calculation, we also have
\[
\begin{split}
(E_\al)_0 e^{i,j} & =(e^{\iota_1(\al)} + \cdots +e^{\iota_n(\al)})_0 \left( \frac{1}{16} \omega_{M_{i,j}} +\frac{1}{32} \sum_{\al\in \Phi^+(E_8)} (e^{\iota_i(\al)- \iota_j(\al)}  + e^{\iota_j(\al)- \iota_i(\al)})\right)\\
&= \frac{1}{16} \left(\la \al, \al\ra^2 \frac{1}8(\iota_i(\al)(-1)e^{\iota_i(\al)} + \iota_j(\al) e^{\iota_j(\al)})\right. \\
&\left. \quad - 2\la \al, \al\ra \frac{1}{8}( (\iota_i(\al)-\iota_j(\al) )e^{\iota_i(\al)}
- (\iota_i(\al)-\iota_j(\al) )e^{\iota_j(\al)})\right)\\
& \quad  +
\frac{1}{32}\varepsilon(\al, -\al)  (\iota_j(\al)(-1) e^{\iota_i(\al)} +\iota_i(\al)(-1) e^{\iota_j(\al)})\\
&=0.
\end{split}
\]
Note that the central charge of  $\tilde{W} = 8n - \frac{n(248)}{n+30}= \frac{8(n^2-n)}{n+30},$ which is the central charge of $ \W_{A_{n-1}}  $.
Hence $ \W_{A_{n-1}} $ is a full subVOA of $\tilde{W}$.
\end{proof}

\begin{rem}
We believe that the commutant VOA $\tilde{W}$ is generated by its weight two subspace. If it is true, then $\tilde{W}=\W_{A_{n-1}}$.
\end{rem}

\subsection{Group of the shape $3^{k}:S_n$}
Next, we shall construct a VOA generated by Ising vectors
such that the subgroup generated by the corresponding  Miyamoto involutions has the shape $3^{n-2}:S_n$ if $n=0\mod 3$ or $3^{n-1}:S_n$ if $n\neq 0 \mod 3$. A special case when $n=3$ has been discussed in \cite{CL3c}.

\medskip

We continue to use the notation in Section \ref{sec:3} and Lemma \ref{lem:5.2}.

\begin{df}\label{rhoi}
Let $\delta \in E_8$ such that
 \[
K:= \{\be\in E_8\mid \la \be, \delta\ra \in 3\Z\} \cong A_8.
 \]
Define an automorphism $\rho_i$ of $V_X$ by
\[
\rho_i = \exp\left( \frac{2\pi \sqrt{-1} }3 \iota_i(\delta) (0)\right).
\]
Then $\rho _i$ has order $3$ and the fixed point subspace $(V_{M_{i,j}})^{\rho_i} \cong V_{\sqrt{2}A_8}$.
\end{df}

\begin{nota}
Let $\mathcal{I} = \{ \rho_i^\ell (e^{i,j})\mid 1\leq i< j\leq n, \ell =0,1,2\}$.
Then $\mathcal{I}$ is a set of Ising vectors in $V_X$. Note that $\rho_i^\ell (e^{i,j}) = \rho_j^{-\ell} (e^{i,j})$ for any $i,j=1, \dots,n$ and $\ell=0,1,2$.
\end{nota}

\begin{lem} \label{inn2}
Let $i,j,k,r, p,q \in \{1, \dots, n\}$ and $\ell, s =0,1,2$. Then
\[
\la \rho_k^\ell (e^{i,j}) ,\rho_r^s (e^{p,q})\ra =
\begin{cases}
0 & \text{ if } \{i,j\} \cap \{p,q\}=\emptyset, \\
\frac{1}{2^8} & \text{ if either }  \#(\{i,j\} \cap \{p,q\})=1, \text{ or }\\
              & \quad \{i,j\}=\{p,q\} \text{ but } \rho_k^\ell (e^{i,j})\neq \rho_r^s(e^{p,q}), \\
\frac{1}4     & \text{ if }   \rho_k^\ell (e^{i,j})= \rho_r^s(e^{p,q}).
\end{cases}
\]
\end{lem}

\begin{proof}
See $3C$ case of \cite{LYY} and Lemma \ref{inner}.
\end{proof}

\begin{lem}\label{rho}
Let $\{i,j\}\subset \{1,\dots,n\}$ and $\ell=0,1,2$. Then
$\tau_{\rho_i^\ell(e^{i,j})} \tau_{e^{i,j}} = \rho_i^\ell \rho_j^{-\ell}$ as an automorphism of $V_X$.
\end{lem}

\begin{proof}
By Lemma \ref{tauM} and Corollary \ref{transij}, $\tau_{e^{i,j}}$ acts as the transposition $(i,j)$ on the tensor copies of the VOA $V_X\cong V_{E_8}\otimes \cdots\otimes V_{E_8}$.

Since $\rho_i$ is trivial on $M(1)$ and $\tau_{e^{i,j}}(M(1)) \subset M(1)$,  we have
\[
\tau_{\rho_i^\ell(e^{i,j})} \tau_{e^{i,j}} u = \rho_i^\ell \tau_{e^{i,j}} \rho_i^{-\ell} \tau_{e^{i,j}} u= \rho_i^\ell \tau_{e^{i,j}} \tau_{e^{i,j}} u = u \quad \text{ for all } u\in M(1).
\]
Moreover,
\[
\tau_{\rho_i^\ell(e^{i,j})} \tau_{e^{i,j}} e^{(x_1, \dots, x_n)} =
 \rho_i^\ell \tau_{e^{i,j}} \rho_i^{-\ell} \tau_{e^{i,j}} e^{(x_1, \dots, x_n)} =
 \exp\left(\frac{2\pi \sqrt{-1} }3 \ell (\la x_i, \delta \ra - \la x_j, \delta\ra )\right) e^{(x_1, \dots, x_n)},
\]
where $(x_1, \dots, x_n)\in E_8^n$. Thus, we have $\tau_{\rho_i^\ell(e^{i,j})} \tau_{e^{i,j}} = \rho_i^\ell \rho_j^{-\ell}$ as an automorphism of $V_X$.
\end{proof}

\begin{lem}
The set $\mathcal{I}$ satisfies the $(2B,3C)$-condition.
\end{lem}

\begin{proof}
By Lemma \ref{inn2}, it suffices to show that $VOA(e,f)\cong U_{3C}$ if $\la e,f\ra=1/2^8$. On the other hand, by Lemma \ref{rho}, we know that
$\tau_{\rho_i^\ell (e^{i,j})}\tau_{\rho_i^s(e^{i,j})}$ has order $3$ or $1$. Moreover,
\begin{equation}\label{com}
\rho_i^\ell \tau_{e^{i,j}} = \tau_{e^{i,j}} \rho_j^\ell\qquad \text{ for any } i,j, \ell.
\end{equation}
Therefore,
\[
\begin{split}
(\tau_{\rho_i^\ell (e^{i,j})}\tau_{ e^{i,k}})^3
&= (\rho_i^\ell\rho_j^{-\ell}  \tau_{e^{i,j}}\tau_{e^{i,k}})^3\\
&= \rho_i^\ell\rho_j^{-\ell}  \tau_{e^{i,j}}\tau_{e^{i,k}}\rho_i^\ell\rho_j^{-\ell}  \tau_{e^{i,j}}\tau_{e^{i,k}}\rho_i^\ell\rho_j^{-\ell}  \tau_{e^{i,j}}\tau_{e^{i,k}}\\
&= \rho_i^\ell\rho_j^{-\ell}  \rho_k^\ell\rho_i^{-\ell} \tau_{e^{i,j}}\tau_{e^{i,k}} \tau_{e^{i,j}}\tau_{e^{i,k}}\rho_i^\ell\rho_j^{-\ell}  \tau_{e^{i,j}}\tau_{e^{i,k}}\\
&= \rho_i^\ell\rho_j^{-\ell}  \rho_k^\ell\rho_i^{-\ell} \tau_{e^{i,j}}\tau_{e^{i,k}} \tau_{e^{i,j}}\tau_{e^{i,k}} \tau_{e^{i,j}}\tau_{e^{i,k}}\rho_j^\ell\rho_k^{-\ell} \\
&= \rho_i^\ell\rho_j^{-\ell}  \rho_k^\ell\rho_i^{-\ell} \rho_j^\ell\rho_k^{-\ell} =1.
\end{split}
\]
Similarly, $\tau_{\rho_i^\ell (e^{i,j})}\tau_{ \rho_i^s(e^{i,k})}$ also has order $3$.
\end{proof}

\medskip

Recall that $K= \{\be\in E_8\mid \la \be, \delta \ra \in 3\Z\} \cong A_8$. Then $\{ H_\al, E_\al \mid \al\in K \text{ is a root}\}$
generates  a subVOA isomorphic to $L_{A_8}(n,0)$, which is an irreducible level $n$ representation of $\hat{sl}_9(\C)$.

\begin{prop}
Let $\tilde{\W}:=\Com_{V_X}(L_{A_8}(n,0))$. Then $\mathcal{I} \subset \tilde{\W} $.
Moreover, the central charge of $\tilde{\W}$ is $8n(n-1)/(n+9)$.
\end{prop}

\begin{proof}
The proof is similar to Lemma \ref{lem:5.2}.  Since $\mathcal{I} \subset V_{\mathcal{A}_{n-1}}$,
it is clear that $(H_\al)_n e =0 $ for all $n\geq 0$ and $e\in \mathcal{I}$.

By Lemma \ref{lem:5.2}, it is clear that $(E_\al)_n e^{i,j} =0$ for all $n\geq 0$ and $i,j\in \{1, \dots, n\}$. It is also clear that $(E_\al)_n \rho_i^k e^{i,j} =0$ for any $n\geq 2$ and $k=1,2$.

For any root $\al\in K$, we have
\[
\begin{split}
(E_\al)_1 \rho_i^k e^{i,j} & =(e^{\iota_1(\al)} + \cdots +e^{\iota_n(\al)})_1 \left( \frac{1}{16} \omega_{M_{i,j}} +\frac{1}{32} \sum_{\al\in \Phi^+(E_8)} \rho_i^k(e^{\iota_i(\al)- \iota_j(\al)}  + e^{\iota_j(\al)- \iota_i(\al)})\right)\\
&= \frac{1}{16} \left(\la \al, \al\ra^2 \frac{1}8(e^{\iota_i(\al)} + e^{\iota_j(\al)})\right) +
\frac{1}{32}\varepsilon(\al, -\al)  (e^{\iota_i(\al)} + e^{\iota_j(\al)}) =0,
\end{split}
\]
and
\[
\begin{split}
(E_\al)_0 \rho_i^k e^{i,j} & =(e^{\iota_1(\al)} + \cdots +e^{\iota_n(\al)})_0 \left( \frac{1}{16} \omega_{M_{i,j}} +\frac{1}{32} \sum_{\al\in \Phi^+(E_8)} \rho_i^k(e^{\iota_i(\al)- \iota_j(\al)}  + e^{\iota_j(\al)- \iota_i(\al)})\right)\\
&= \frac{1}{16} \left(\la \al, \al\ra^2 \frac{1}8(\iota_i(\al)(-1)e^{\iota_i(\al)} + \iota_j(\al) e^{\iota_j(\al)})\right. \\
&\left. \quad - 2\la \al, \al\ra \frac{1}{8}( (\iota_i(\al)-\iota_j(\al) )e^{\iota_i(\al)}
- (\iota_i(\al)-\iota_j(\al) )e^{\iota_j(\al)})\right)\\
& \quad  +
\frac{1}{32}\varepsilon(\al, -\al)  (\iota_j(\al)(-1) e^{\iota_i(\al)} +\iota_i(\al)(-1) e^{\iota_j(\al)})\\
&=0.
\end{split}
\]
Therefore, $(E_\al)_n  \rho_i^k e^{i,j} =0$ for all $n\geq 0$.
\end{proof}

\begin{prop}
The commutant subVOA $\tilde{\W}=\Com_{V_X}(L_{A_8}(n,0))$ has zero weight one subspace.
\end{prop}

\begin{proof}
The proof is similar to Theorem 5.4 of \cite{CL3c}.

Since $h(-1)\cdot \mathbbm{1} \in L_{A_8}(n,0)$ for all $h\in d(E_8)$, the commutant subVOA
\[
 \mathrm{Com}_{V_X}\left( L_{A_8}(n,0)\right) \subset V_{\mathcal{A}_{n-1}}.
\]
Therefore, it suffices to show $\tilde{\W} \cap ( V_{\mathcal{A}_{n-1}})_1 =0$.

Since $\mathcal{A}_{n-1}$ has no roots, we have
\[
( V_{\mathcal{A}_{n-1}})_1 =span\{h(-1)\cdot \mathbbm{1}\mid h\in  \mathcal{A}_{n-1}\}.
\]
Let $h=\sum_{i=1}^n a_i \iota_i(\be_i)\in \mathcal{A}_{n-1}$. Then
\[
(E_\al)_0 h(-1)\cdot \mathbbm{1} = \sum_{i=1}^n a_i  \la \al, \be_i\ra e^{\iota_i(\al)}.
\]
Suppose $h(-1)\cdot \mathbbm{1} \in \tilde{\W}$. Then  we have $a_i\la \al, \be_i\ra =0$ for all root $\al \in K$.  For $a_i\neq 0$, we must have $\la \al, \be_i\ra =0$ for all root $\al \in K$. Thus $\la \be_i, K\ra =0$ and $\be_i=0$ since $K$ is a full rank sublattice of $E_8$. Hence $h=0$ as desired.
\end{proof}

\begin{prop}
Let $\EuScript{G}$ be the subgroup generated by $\{\tau_e\mid e\in\mathcal{I}\}$. Then  as an automorphism subgroup of $\W$,
$\EuScript{G}$ has the shape $3^{n-2}{:}S_n$ if $n=0\mod 3$ and $3^{n-1}{:}S_n$ if $n\neq 0\mod 3$.
\end{prop}

\begin{proof}
By Theorem \ref{GR} (see also Corollary \ref{transij}), the subgroup generated by $$\{\tau_{e^{i,j} } \mid 1\leq i<j\leq n\}$$ is isomorphic to $S_n$.  By Lemma \ref{rho}, we have an elementary abelian $3$-group $H$ generated by $\rho_i\rho_j^{-1}$ for $i,j\in \{1, \dots,n\}$. It is also straightforward to show by the same arguments as in Lemma \ref{rho} (see also Equation \eqref{com}) that $\EuScript{G}$ has the shape $H: S_n$.

It remains to show that the order of $H$ is
\[
\begin{cases}
3^{n-2} & \text{ if } n=0\mod3, \\
3^{n-1}& \text{ if } n\neq 0 \mod 3.
\end{cases}
\]

Let $h=\rho_1^{a_1} \cdots \rho_n^{a_n}\in H$. Then $a_1+\cdots +a_n=0$.
Suppose $h=id $ on $V_{\mathcal{A}_{n-1}}$.
Then we have $a_i-a_j =  0 \mod 3$ for all $i,j$. It implies
\[
0=a_1+\cdots +a_n =  n a_i \mod 3.
\]
That means $n  =  0\mod 3$ or $a_1 = \dots = a_n = 0 \mod 3$.

Therefore, $|H|=3^{n-1}$ if $n\neq 0\mod 3$ and $|H|= 3^{n-2} $ if $n=0\mod 3$.
\end{proof}

The next proposition is essentially proved in \cite[Theorems 6.11 and 6.12]{CL3c} with
some minor modifications.

\begin{prop}
The commutant VOA $\tilde{\W} =\Com_{V_X}(L_{A_8}(n,0))$ is generated by $\mathcal{I} = \{ \rho_i^\ell (e^{i,j})\mid 1\leq i< j\leq n, \ell =0,1,2\}$. In particular, $\W=\tilde{\W}$.
\end{prop}



\begin{thebibliography}{ams}

\bibitem[ATLAS]{ATLAS}
  J.H. Conway, R.T. Curtis, S.P. Norton, R.A. Parker and  R.A. Wilson,
  ATLAS of finite groups. Clarendon Press, Oxford, 1985.

\bibitem[CL]{CL3c} H.~Y. Chen and C.~H. Lam,  On Majorana representations of the group $3^2{:}2$ of $3C$-pure type
and the corresponding vertex operator algebra, preprint.


\bibitem[CS]{CS}
J.\,H.\,Conway and N.\,J.\,A. Sloane, Sphere packings, lattices and
groups. 3rd Edition, Springer, New York, 1999.

\bibitem[DL]{DL}
 C. Dong and J. Lepowsky, Generalized vertex algebras and relative
  vertex operators, Progress in Math. {\bf 112}, Birkh\"{a}user,
  Boston, 1993.

\bibitem[DLMN]{dlmn}  C. Dong, H. Li, G. Mason and S.P. Norton, Associative
subalgebras of Griess algebra and related topics, \emph{Proc. of
the Conference on the Monster and Lie algebra at the Ohio State
University}, May 1996, ed. by J. Ferrar and K. Harada, Walter de
Gruyter, Berlin - New York, 1998.


\bibitem[DMZ]{dmz} C. Y. Dong, G. Mason, and Y. Zhu, Discrete series
of the Virasoro algebra and the Moonshine module, \emph{Proceedings
of Symposia in Pure Mathematics }\textbf{56}\emph{ Part 2} (1994)
295-316.



\bibitem[FHL]{FHL} I. Frenkel, Y. Z. Huang, and J. Lepowsky, On axiomatic
approaches to vertex operator algebras and modules, \emph{Mem.  Amer. Math. Soc.} \textbf{104}, 1993.

\bibitem[FLM]{FLM} I. Frenkel, J. Lepowsky, and A. Meurman, Vertex
operator algebras and the Monster, \emph{Academic Press, New York},
1988.

\bibitem[FZ]{FZ} Igor B. Frenkel, and Y. C. Zhu, Vertex operator
algebras associated to representations of affine and Virasoro algebras, \emph{Duke Math.  J.} \textbf{66} (1992) 123-168.

\bibitem[G]{G}
  R.L. Griess, The friendly giant, {\it Invent. Math.} {\bf 69}
  (1982), 1--102.

\bibitem[GL]{GL}
  R.L. Griess and C. H. Lam, $EE_8$ lattices and dihedral groups,
  Pure and Applied Math. Quarterly (special issue for Jacques Tits),
  {\bf 7} (2011), no. 3, 621-743. {\tt arXiv:0806.2753}.

  \bibitem[GL1]{GL3c}  R.\,L.\,Griess, Jr. and C.\,H.\,Lam,  A moonshine path from $E_8$ to the Monster, J. Pure and Applied Algebra, 215(2011), 927-948.


\bibitem[H\"o]{ho} G. H\"{o}hn, The group of symmetries of the shorter moonshine module.
  {\it Abhandlungen aus dem Mathematischen Seminar der Universit\"at Hamburg}
 {\bf 80} (2010), 275--283, {\tt arXiv:math/0210076}.

\bibitem[Iv]{Iva} A. A. Ivanov, \emph{The Monster group and Majorana involutions}, Cambridge Tracts in Mathematics 176,
Cambridge University Press, Cambridge, 2009. xiv+252 pp.





\bibitem[Iv2]{Iva2} A. A. Ivanov, Majorana representation of $A_6$ involving $3C$-algebras, \emph{ Bull. Math. Sci.} \textbf{1} (2011), no. 2, 365-378.

\bibitem[Iv3]{Iva3} A. A. Ivanov, On Majorana representations of $A_6$ and $A_7$, \emph{Comm. Math. Phys.} \textbf{307} (2011), no. 1, 1-16.


\bibitem[IS]{IS}  A.A. Ivanov and A. Seress, Majorana representations of $A_5$, \emph{Math. Z.}  \textbf{272} (2012), no. 1-2, 269-295.

\bibitem[IPSS]{IPSS} A. A. Ivanov, D. V. Pasechnik, A. Seress, and
S. Shpectorov, Majorana representations of the symmetric group of degree 4, \emph{J.  Algebra} \textbf{324} (2010) 2432-2463.


\bibitem[LL]{LL}
J.~Lepowsky and H.~Li, \emph{Introduction to vertex operator algebras and their
  representations}, Progress in Mathematics, vol. 227, Birkh{\"a}user, Boston,
  2003.

\bibitem[LSa]{LSa}
  C.H. Lam and  S. Sakuma, On a class of vertex operator algebras having a faithful -action, \emph{Taiwanese J. Math}, \textbf{12} (2008), no. 9, 2465-2488.

\bibitem[LSY]{LSY}
  C.H. Lam, S. Sakuma and H. Yamauchi, Ising vectors and automorphism
  groups of commutant subalgebras related to root systems,
  {\it Math. Z.} {\bf 255} vol.3 (2007), 597--626.

\bibitem[LS]{LS} C.H. Lam and H. Shimakura, Ising vectors in the vertex operator algebra $V^+_\Lambda$ associated with the Leech lattice $\Lambda$, \emph{Int. Math. Res. Not.}  IMRN 2007, no. 24, Art. ID rnm132, 21 pp.


\bibitem[LYY1]{LYY} C. H. Lam, H. Yamada, H. Yamauchi, Vertex operator
algebras, extended $E_{8}$ diagram, and McKay's observation on the Monster simple group, \emph{Trans.  Amer. Math.
Soc.}\textbf{359} (2007) 4107-4123.

\bibitem[LYY2]{LYY2}
  C.H. Lam, H. Yamada and H. Yamauchi, McKay's observation and vertex
  operator algebras generated by two conformal vectors of central
  charge 1/2. {\it Internat. Math. Res. Papers} {\bf 3} (2005), 117--181.


\bibitem[Mi]{miy} M. Miyamoto, Griess algebras and conformal vectors
in vertex operator algebras, \emph{J.  Algebra }\textbf{179 }(1996) 523-548.

\bibitem[Mi2]{M} M. Miyamoto, A new construction of the moonshine
vertex operator algebra over the real number field, \emph{Ann.  Math. } \textbf{159} (2004) 535-596.


\bibitem[Sa]{Sa} S. Sakuma, 6-transposition property of $\tau$-involutions
of vertex operator algebras, \emph{Inter. Math. Res. Not.} \textbf{2007}, Article ID rnm 030, 19 pages.


\bibitem[Sh]{Sh}  H. Shimakura, The automorphism group of the vertex operator algebra $V^+_L$ for an even lattice L without roots, \emph{J. Algebra} \textbf{280} (2004), no. 1, 29 - 57.

\bibitem[Sh2]{ShI}H. Shimakura, Classification of Ising vectors in the vertex operator algebra $V^+_L$, \emph{Pacific J. Math.}  \textbf{258} (2012), no. 2, 487- 495.

\end{thebibliography}
\end{document}